\documentclass[a4paper, 11pt,psamsfonts]{amsart}
\usepackage{amsmath}
\usepackage{amsthm}
\usepackage{amssymb}
\usepackage{amscd}
\usepackage{amsfonts}
\usepackage{amsbsy}
\usepackage{enumerate}
\usepackage{graphicx}
\usepackage{calc}
\usepackage{verbatim}
\usepackage{xcolor}
\usepackage{setspace}

\newtheorem {theorem} {Theorem}
\newtheorem {proposition}{Proposition}

\newtheorem {lemma}{Lemma}
\newtheorem {definition} {Definition}
\newtheorem {remark}{Remark}
\newtheorem {example} {Example}

\begin{document}
\title
{Coordinate-independent criteria for Hopf bifurcations}

\author
{Niclas Kruff and Sebastian Walcher}

\address{Lehrstuhl A f\"ur Mathematik, RWTH
Aachen, 52056 Aachen, Germany}

\begin{abstract}
We discuss the occurrence of Poincar\'e-Andronov-Hopf bifurcations in parameter dependent ordinary differential equations, with no a priori assumptions on special coordinates. The first problem is to determine critical parameter values from which such bifurcations may emanate; a solution for this  problem was given by W.-M.~Liu. We add a few observations from a different perspective. Then we turn to the second problem, viz., to compute the relevant coefficients which determine the nature of the Hopf bifurcation. As shown by J.~Scheurle and co-authors, this can be reduced to the computation of Poincar\'e-Dulac normal forms (in arbitrary coordinates) and subsequent reduction, but feasibility problems quickly arise. In the present paper we present a streamlined and less computationally involved approach to the computations. The efficiency and usefulness of the method is illustrated by examples.\\
{{\bf MSC} (2010): 34C20, 34C23, 37G15.}
\end{abstract}

\maketitle

\section{Introduction}
Poincar\'e-Andronov-Hopf bifurcations (briefly called Hopf bifurcations in the present paper) frequently occur in parameter dependent ordinary differential equations, with numerous applications. A classical comprehensive source is Marsden and McCracken \cite{MaMc}. Among the monographs on differential equations and dynamical systems which discuss Hopf bifurcations we mention only a few, viz.  Guckenheimer and Holmes \cite{GuHo}, Thm.~3.4.2, Amann \cite{Amann}, Thm. 26.21, Aulbach \cite{Aulbach}, Satz 7.11.1 and  Chicone \cite{Chi}, Section 8.3 (in particular Thm.~8.25). \\ The familiar statement of the Hopf bifurcation theorem assumes that the system is given in a standardized form which includes (i) a distinguished real bifurcation parameter as well as (ii) a convenient choice of coordinates. Given a suitable real bifurcation parameter, one coefficient in the Poincar\'e-Dulac normal form (which is rational in the coefficients of the Taylor expansion) determines the nature of the bifurcation (subcritical, supercritical or degenerate). The computation of this coefficient is unproblematic in dimension two (and fairly easy in higher dimensions) if the linearization at the stationary point is given in (real) Jordan canonical form.  But for systems that are not normalized, technical problems in the practical implementation arise in several ways. For multi-parameter systems, prior to determining a single parameter so that (i) is satisfied, one has to address the problem of finding the critical parameter values from which Hopf bifurcations emanate (for some choice of a curve through this point in parameter space). This part of the problem is concerned only with the linearization of the vector field, and has been resolved by W.-M.~Liu  \cite{Liu} based on the classical Routh-Hurwitz criteria. For (polynomial) systems which model chemical reaction networks Liu's ansatz taken further by Errami et al. \cite{EEGSSW} who devised an algorithm to find all possible stationary points and critical parameter values for Hopf bifurcations. But Errami et al. did not proceed to determine the nature of the bifurcations. This part of the problem involves nonlinear terms in the Taylor expansion, and this is the part we mainly address in the present paper, with Poincar\'e-Dulac normal forms as the fundamental tool.\\
Normal forms, and the ensuing reductions, were discussed in coordinate-independent settings (including an algorithm for their computation) by Scheurle and Walcher \cite{SW} and Mayer, Scheurle and Walcher \cite{MSW}. In principle these results may be used for the necessary computations to determine the nature of a Hopf bifurcation, but this general approach has the drawback of high computational expense. \\
In the present paper we introduce a more efficient approach to compute the necessary coefficients for the discussion of the nature of a Hopf bifurcation. We first give a brief review of the Hopf bifurcation theorem from the perspective of normal forms and then turn to computations in arbitrary coordinates. We  start by recalling Liu's \cite{Liu} results to determine critical parameter values for Hopf bifurcations, and add a few observations. Then, in the central part (4.2--4.4), we turn to determining the relevant coefficients for a Hopf scenario without computing a full normal form. Finally we present a variant of the approach by  Errami et al. \cite{EEGSSW} that is applicable to arbitrary polynomial systems. To show applicability, we discuss some examples, including a FitzHugh-Nagumo equation and a three-dimensional predator-prey system.

\section{Notation and preliminaries}
We consider a parameter dependent differential equation
\begin{equation}\label{odegen}
\dot x=h(x,\,p)
\end{equation}
defined for $(x,p)$ in some nonempty and open subset of $\mathbb R^n\times\mathbb R^m$, and $h$ of class $\mathcal C^4$. We assume throughout that $x=0$ is a stationary point; thus
\[
h(0,p)=0 \quad\text{ for all   }p.
\]
Our goal is to identify parameter values $p^*$ from which Hopf bifurcations emanate, and to determine the nature of these bifurcations. In general the parameter space will have dimension greater than one, while the Hopf bifurcation theorem refers to a single real parameter. Thus we require that for suitable $v\in \mathbb R^m$ the system
\begin{equation}\label{odeeps}
\dot x=\widehat h(x,\varepsilon):=h(x,\,p^*+\varepsilon v)
\end{equation}
undergoes a Hopf bifurcation at $\varepsilon=0$. (In an apparently more general approach, one could consider curves $\gamma(\varepsilon)=p^*+\varepsilon v +o(\varepsilon)$ in parameter space, but only first order terms in $\varepsilon$ turn out to be relevant.)\\
We recall the notion of Lie derivative: Given an autonomous equation $\dot x=g(x)$ on an open  $U\subseteq \mathbb R^n$ and a differentiable $\sigma:\, U\to \mathbb R$, the Lie derivative $L_g(\sigma)$ of $\sigma$ with respect to $g$ is defined by
\[
L_g(\sigma)(x)=D\sigma(x) \,g(x).
\]
(For parameter dependent equations such as \eqref{odegen}, we will always consider the Lie derivative with respect to $x$.) Lie derivatives will play a central role in the following.
\begin{remark}\label{skrem}The Lie derivative of a linear vector field is well understood; see e.g. \cite{MSW}:
Given a linear map $x\mapsto Bx$, the Lie derivative $L_B$ acts on polynomials $\psi:\, \mathbb R^n\to\mathbb R$, and sends the subspace $S_k$ of homogeneous polynomials of degree $k$ to itself, for every $k\geq 1$.
If $\lambda_1,\ldots,\lambda_n$ are the eigenvalues of $B$ (in the complexification, counted with multiplicity) then the eigenvalues of $L_{B}|_{S_k}$ are
\[
\sum_{j=1}^k k_j\lambda_j, \text{ with nonnegative integers  }k_j \text{  and  } \sum k_j=k.
\]
The eigenspaces of $L_{B}|_{S_1}$ are spanned by linear forms $\mu$ which correspond to the left eigenvectors of $B$. The eigenspaces of $L_{B}|_{S_k}$ for $k>1$ are spanned by products of those linear forms.
\end{remark}
In our setting the map $B$ will be the linearization of \eqref{odeeps} at the stationary point $0$, hence $B=B(p)$ is parameter dependent, and in general the eigenvalues cannot be computed explicitly. But the characteristic polynomial (and the minimal polynomial) of $B$ can be determined, and it coincides with the characteristic polynomial (resp. the minimal polynomial) of $L_{B}|_{S_1}$. Starting from this, one can recursively determine polynomials which annihilate $L_{B}|_{S_k}$ for $k>1$; see \cite{MSW}, with some improvements in \cite{Kru}. In turn, knowledge of an annihilating polynomial will allow to solve a system of linear equations, or to compute projections onto certain subspaces, even in the presence of parameters.

\section{The bifurcation scenario in suitable coordinates}
In this section we review some known results. We first consider bifurcations from the perspective of Poincar\'e-Dulac normal forms, referring to Bibikov \cite{Bib} and in particular his notion of normal form on an invariant manifold (NFIM, called NFIS in \cite{Bib}). 
Proofs of the Hopf bifurcation theorem need not explicitly rely on Poincar\'e-Dulac normal forms, but these provide insight into their structure and relevant parameters; see Bibikov \cite{Bib}.
We assume that the linearization of \eqref{odeeps} at  $\varepsilon=0$  has a pair of purely imaginary eigenvalues $\pm i\omega\neq 0$, while the remaining eigenvalues have real part $<0$. Moreover there exist coordinates such that the matrix
has the form
\[
\begin{pmatrix} 0&-\omega&0&\cdots &0\\
                        \omega&0&0&\cdots &0\\
                         0&0&*&\cdots &*\\
                          \vdots &\vdots &\vdots &  & \vdots\\
    0&0&*&\cdots &*\\
\end{pmatrix}.
\]
By normal form theory (more precisely, as a consequence of \cite{Bib}, Thm.~3.1), after the introduction of an additional variable $x_0=\varepsilon$ and a further (near-identity) coordinate transformation, the subspace defined by $x_3=\cdots =x_n=0$ is invariant for the Taylor expansion up to degree three, and the Taylor approximation is in NFIM up to degree three, thus we have
\begin{equation}\label{nfim}
\begin{array}{rcccl}
\dot x_0&=& 0 && \\
\dot x_1&=&-\omega x_2&+&x_0(\alpha x_1-\mu x_2)+(x_1^2+x_2^2)(\beta  x_1-\gamma x_2)+x_0^2(\rho x_1-\sigma x_2)\\
\dot x_2&=&\omega x_1&+&x_0(\alpha x_2+\mu x_1)+(x_1^2+x_2^2)(\beta  x_2+\gamma x_1)+x_0^2(\rho x_2+\sigma x_1)\\
\dot x_3&=&\cdots  && \\
    &\vdots&  && \\
\dot x_n&=&\cdots  && \\
\end{array}
\end{equation}
omitting terms of higher degree.
This system allows reduction: Letting $\phi_0:=x_0$ and $\phi_1:=x_1^2+x_2^2$, the map $\Phi=(\phi_0,\phi_1)^{\rm tr}$ sends solutions of this NFIM to solutions of the two-dimensional system
\begin{equation}\label{redeq}
\begin{array}{rcl}
\dot y_0&=& 0\\
\dot y_1&=& 2(\alpha +\rho y_0) y_0y_1+2\beta y_1^2
\end{array}
\end{equation}
From the reduced equation one determines the nature of the bifurcation as $x_0$ crosses $0$; see e.g. Marsden and McCracken \cite{MaMc} for proofs. Note that only solutions with $y_1\geq 0$ are of interest for \eqref{nfim}, since $\phi_1(x)\geq 0$ for all $x$. For easy reference we enumerate all possible cases with $\alpha\not=0$ and $\beta\not=0$. (The list is twice as long as usual, since we do not prescribe ``crossing the imaginary axis from left to right'' for a pair of conjugate eigenvalues.)
\begin{lemma}\label{hopflem} Assume that $\alpha\not=0$ and $\beta\not=0$. Then the following scenarios occur.
\begin{enumerate}[(a)]
\item Given $y_0>0$ such that $\alpha$ and $\alpha +\rho y_0$ have the same signs (which is the case for sufficiently small $y_0$), consider the second equation of \eqref{redeq}.
\begin{enumerate}[(i)]
\item In case $\alpha>0$ the stationary point $0$ is repelling. 
\begin{itemize}
\item When $\beta>0$ there is no further stationary point on the half-line $y_1>0$.
\item When $\beta<0$ there is a unique attracting stationary point at $y_1=-(\alpha+\rho y_0)y_0/\beta>0$, which gives rise to a unique attracting limit cycle for \eqref{nfim}.
\end{itemize}
\item In case $\alpha<0$ the stationary point $0$ is attracting.
\begin{itemize}
\item When $\beta<0$ there is no further stationary point on the half-line $y_1>0$.
\item When $\beta>0$ there is a unique repelling stationary point at $y_1=-(\alpha+\rho y_0)y_0/\beta>0$, which gives rise to a unique repelling limit cycle for \eqref{nfim}.
\end{itemize}
\end{enumerate}
\item  Given $y_0<0$ such that $\alpha$ and $\alpha +\rho y_0$ have the same signs (which is the case for sufficiently small $|y_0|$), consider the second equation of \eqref{redeq}.
\begin{enumerate}[(i)]
\item In case $\alpha>0$ the stationary point $0$ is attracting.
\begin{itemize}
\item When $\beta<0$ there is no further stationary point on the half-line $y_1>0$.
\item When $\beta>0$ there is a unique repelling stationary point at $y_1=-(\alpha+\rho y_0)y_0/\beta>0$, which gives rise to a unique attracting limit cycle for \eqref{nfim}.
\end{itemize}
\item In case $\alpha<0$ the stationary point $0$ is repelling.
\begin{itemize}
\item When $\beta>0$ there is no further stationary point on the half-line $y_1>0$.
\item When $\beta<0$ there is a unique attracting stationary point at $y_1=-(\alpha+\rho y_0)y_0/\beta>0$, which gives rise to a unique attracting limit cycle for \eqref{nfim}.
\end{itemize}
\end{enumerate}
\end{enumerate}
\end{lemma}
The Hopf bifurcation theorem (roughly speaking) states that the higher order terms in the NFIM do not affect the local qualitative behavior. We do not state it in full generality here, but just give a basic version:
\begin{theorem}\label{hbt}
Let \eqref{nfim} be a NFIM of system \eqref{odeeps} up to degree $3$, with $\alpha\not=0$, $\beta\not=0$. Then for sufficiently small $\varepsilon=x_0$, the statements in Lemma \ref{hopflem} about stationary points and limit cycles of \eqref{nfim} continue to hold for system \eqref{odeeps} in some neighborhood of $0$.
\end{theorem}
Thus, the qualitative study of \eqref{odeeps} near $\varepsilon=0$ amounts to computing the coefficients $\alpha$ and $\beta$
and checking the list in Lemma \ref{hopflem} (unless a degenerate case with $\alpha=0$ or $\beta =0$ occurs). However, these computations pose a nontrivial obstacle, and this is the motivation for the present work. \\
For systems in suitable coordinates a very detailed account of the normal form computations was given in Marsden and McCracken, \cite{MaMc}, Section 4A (see also Section 5 for a translation of Hopf's original work); later improvements and shortcuts can be found e.g. in Hassard and Wan \cite{HaWa} (in dimension two), and in \cite{WaNFT} (in arbitrary dimension) for the case that the linearization is in (real or complex) Jordan form. \\
But for general systems, an exact determination of eigenvalues or eigenspaces is not possible (a fortiori so when the system is parameter dependent). 
In this sitiuation a coordinate-free approach to compute Poincar\'e-Dulac normal forms (as presented by Scheurle and co-authors in \cite{SW} and \cite{MSW}), followed by a reduction (analogous to passing from \eqref{nfim} to \eqref{redeq}; see \cite{MSW}) opens a path, in principle. (An illustrative example is given in \cite{MSW}, Section 6.) However, feasibility problems quickly arise in higher dimensions, and moreover the normal form contains more information than necessary to compute the coefficients $\alpha,\,\beta$ in \eqref{nfim}. We will therefore show how to streamline the computations. The motivation for the shortcut is the determination of  Lyapunov's focus quantities in the two-dimensional center problem (see e.g. Artes et al. \cite{DLA} Ch.~4).

\section{The bifurcation scenario in unsuitable coordinates}\label{unsuitsec}
\subsection{Critical parameter values}
The first task is to identify those parameter values $p^*$ of \eqref{odegen} from which Hopf bifurcations may emanate. This amounts to conditions on the eigenvalues of
\begin{equation}\label{odegenlin}
B(p):=Dh(0,p),
\end{equation}
since a Hopf bifurcation emanates from $p=p^*$ only if all eigenvalues of $B(p^*)$  have real part $\leq0$ and there is exactly one complex conjugate pair with real part zero. This problem has been solved by Liu \cite{Liu}, building on classical results by Routh and Hurwitz.\\
We first consider the roots of a general normalized polynomial
\begin{equation}\label{polygen}
\chi(\tau):=\tau^n+ c_1\tau^{n-1}+\cdots + c_n;\quad n\geq 2
\end{equation}
with real coefficients. 
\begin{remark}\label{liurem}
We recall a few (classical and more recent) results on the position of the roots of \eqref{polygen} in the complex plane; see Gantmacher \cite{Gan}, Ch.~V, \S6, and Liu \cite{Liu}.
\begin{itemize}
\item If all roots of $\chi$ have real parts $<0$ then $c_1>0,\ldots,c_n>0$.
\item (Hurwitz/Routh): All zeros of $\chi$ have real parts $<0$ if and only if all Hurwitz determinants $\Delta_1,\ldots,\Delta_{n-1}$ are $>0$.
\item (Extension of the Routh-Hurwitz criterion; Liu \cite{Liu}:)  All roots of $\chi$  have real part $\leq0$ with exactly one complex conjugate pair with real part zero, if and only if
\[
 \Delta_1>0,\ldots,\Delta_{n-2}>0 \text{  and  } \Delta_{n-1}=0.
\]
If the coefficients $c_i$ depend on a real parameter $\varepsilon$, then these conditions at $\varepsilon=0$, together with  $\partial\Delta_{n-1}/\partial\varepsilon|_{\varepsilon=0}\not=0$, are necessary and sufficient for a Hopf bifurcation at $\varepsilon=0$. (Liu requires the derivative to be $>0$, in order to ensure ``crossing from left to right''.)
\end{itemize}
\end{remark}
Although the problem is thus solved in principle, it seems worthwhile to consider the problem from a different perspective.
\begin{lemma}\label{critlem} Let $B=B(p)$ as in \eqref{odegenlin}, and assume that its characteristic polynomial is given by $\chi$ in \eqref{polygen}.
\begin{enumerate}[(a)]
\item Then $B(p)$ is invertible and admits a pair of eigenvalues which add up to zero if and only if $c_n\not=0$ and the characteristic polynomial of $L_{B(p)}|_{S_2}$ has constant coefficient zero.
\item In this case, there is a factorization
\[
\chi(\tau)=\left(\tau^{n-2}+a_1\tau^{n-3}+\cdots + a_{n-2}\right)\cdot\left(\tau^2+b\right).
\]
All roots have real part  $\leq 0$ with exactly one purely imaginary pair if and only if $b>0$ and the first factor satisfies the Hurwitz-Routh conditions.
\item $B(p)$ satisfies Liu's conditions on the roots from Remark \ref{liurem} if and only if
\begin{equation}\label{critcond}
L_{B(p)}|_{S_1}\text{  is injective, } \dim{\rm Ker}\,(L_{B(p)}|_{S_2})=1,
\end{equation}
and the conditions from part (b) hold. In this case the kernel of $L_{B(p)}|_{S_2}$ is spanned by a positive semidefinite quadratic form $\psi$ of rank two.
\end{enumerate}
\end{lemma}
\begin{proof} This follows from Remarks \ref{skrem} and \ref{liurem}.
\end{proof}
\begin{definition}
\begin{enumerate}[(i)]
\item We call $p^*$ a {\em weakly critical parameter value} of system \eqref{odegen} if condition \eqref{critcond} holds at $p=p^*$.
\item If in addition the characteristic polynomial of $B(p^*)$ satisfies the conditions from Lemma \ref{critlem}(b) then we call $p^*$ a {\em critical parameter value} of system \eqref{odegen}.
\end{enumerate}
\end{definition}
We illustrate the conditions of  Lemma \ref{critlem} for small degrees.
\begin{example}\label{polyex}
\begin{itemize}
\item For degree $n=2$, one just has $\chi(\tau)=\tau^2+c_2$, with $c_2>0$.
\item Consider $\chi(\tau)=\tau^3+c_1\tau^2+c_2\tau+c_3$; thus $n=3$. A straightforward computation shows that $\chi$ satisfies the conditions from Lemma \ref{critlem}(a) if and only if $c_3=c_1c_2\not=0$ (or $\Delta_2=0$). This yields a factorization
\[
\chi=(\tau+c_1)(\tau^2+c_2)
\]
and the relevant setting for Hopf bifurcations is determined by $c_1>0$ and $c_2>0$.
\item For $n=4$, hence $\chi(\tau)=\tau^4+c_1\tau^3+c_2\tau^2+c_3\tau +c_4$, the conditions from Lemma \ref{critlem}(a) hold only if
\[
c_1>0, c_4>0 \text{  and  } -c_1^2c_4+ c_1c_2c_3-c_3^2=0\quad( \text{or  } \Delta_3=0)
\]
with factorization
\[
\chi=\left(\tau^2+c_1\tau+(c_2-c_3/c_1)\right)\cdot\left(\tau^2+c_3/c_1\right).
\]
The relevant conditions for Hopf bifurcations hold if and only if the coefficients of both factors are $>0$.
\end{itemize}
\end{example}
At critical parameter values of \eqref{odegen}, the eigenvalues of $B(p^*)$ ensure that a transformation to NFIM \eqref{nfim} up to degree 3 exists for \eqref{odeeps}, given any choice of $v$; there remains to determine $\alpha$ and $\beta$, which will be done in the next subsection. There is a geometric interpretation: In parameter space $\mathbb R^m$, the condition from Lemma \ref{critlem}(a) defines a hypersurface, and the Hurwitz-Routh inequalities from part (b) determine a semi-algebraic subset of this hypersurface. The condition $\alpha\not=0$ in \eqref{nfim} means that for some $v$ the curve $\varepsilon\mapsto p^*+\varepsilon v$ in parameter space crosses the hypersurface transversally. (This corresponds to Liu's partial derivative condition on $\Delta_{n-1}$.)


\begin{remark}\label{liucondrem}
We compare Liu's \cite{Liu} approach to computing critical parameter values with the one given in Lemma \ref{critlem}. Thus let \eqref{polygen} represent the characteristic polynomial of $B=B(p))$.
Liu computes the determinant $\Theta=\Delta_{n-1}$ of the Hurwitz-Routh matrix
\begin{align*}
L_{n}:=
 \begin{bmatrix} c_{n-1} & c_{n} & 0 & 0 & 0 & \cdots & 0 \\ c_{n-3} & c_{n-2} & c_{n-1} & c_{n} & 0 \cdots & 0 \\ \vdots & \vdots & \vdots & \vdots & \vdots & \cdots & \vdots \\ 0 & 0 & 0 & 0 & 0 & \cdots & 1 \end{bmatrix}
\end{align*}
of size $n\times n$, where $c_{n}:=1$ and $c_{i}=0$ for all $i<0$ and $i\geq n+1$. Taking the approach from Lemma \ref{critlem} one determines the characteristic polynomial of $L_{B}|_{S_{2}}$ and considers its constant coefficient.  In terms of eigenvalues (i.e. roots of $\chi$) this constant coefficient is given by 
\[
 \left(\prod\limits_{1\leq i<j\leq n}^{}(\lambda_{i}+\lambda_{j})\right)\cdot \left(\prod\limits_{i=1}^{n}2\lambda_{i}\right).
\]
However, here it suffices to consider the first factor $\widetilde \Theta$ because all eigenvalues are nonzero. (One also obtains a corresponding factorization when the constant coefficient is rewritten in terms of the $c_i$.) One verifies that $\Theta$ and $\widetilde\Theta$ (as polynomials in the $c_i$) are irreducible, have the same degree, and have the same set of complex zeros; hence they coincide by Hilbert's Nullstellensatz, up to a constant factor. (Details will be given in \cite{Kru}.)\\
The approach by Liu is clearly more efficient in terms of computational expense when only critical parameter values are to be determined. However, we will need the annihilating polynomials of $L_B$ on $S_k$ for $k=2,\,3,\,4$ later on. In particular, from the annihilating polynomial of $L_B$ on $S_2$ we find an element in the kernel, which is the quadratic form needed in the reduction.

\end{remark}

\subsection{Restatement of the Hopf conditions}
Given a critical parameter value $p^*$ we introduce the abbreviation
\begin{equation}\label{fdef}
f(x):=h(x,p^*), 
\end{equation}
and furthermore for $(x_0,x)^{\rm tr}\in\mathbb R\times \mathbb R^n$, and $v\in\mathbb R^m$ as in \eqref{odeeps} we set 
\begin{equation}\label{Fdef}
F_v(x_0,x):=\begin{pmatrix} 0\\h(x,\,p^*+x_0 v)\end{pmatrix},\quad C(p^*):=DF_v(0,\,0).
\end{equation}
We will briefly write $F(x_0,x)$ instead of $F_v(x_0,x)$ when the context is clear.
\begin{proposition}\label{hopfconrestate} There is a Hopf bifurcation of system \eqref{odegen} emanating from a parameter $p^*$, for system \eqref{odeeps}, with $\alpha\not=0$ and $\beta\not=0$ in the NFIM \eqref{nfim}, if and only if:
\begin{enumerate}[(i)]
\item The conditions from Lemma \ref{critlem}(c) hold.
\item With $f$ given in \eqref{fdef} and $\psi$ as in Lemma \ref{critlem}(c) there is a smooth function
\[
\widehat\psi=\psi+\text{higher order terms in  }x
\]
such that 
\begin{equation}\label{betacon}
L_f(\widehat\psi)=2\beta \psi^2+\text{h.o.t.}
\end{equation}
\item Letting $F=F_v$ as in \eqref{Fdef}, there is a smooth function
\[
\widehat\theta=\psi+\text{higher order terms in  }(x_0,\,x)
\]
such that 
\begin{equation}\label{alphacon}
L_F(\widehat\theta)=2\alpha x_0\psi+\text{h.o.t.}
\end{equation}
\end{enumerate}
\end{proposition}
\begin{proof}We have
\[
F(0,x)=\begin{pmatrix} 0\\f(x)\end{pmatrix}.
\]
Denote by $\widehat F(x_0,x)$ a NFIM \eqref{nfim}  of $F$ up to degree three.
Then 
part (i) is a direct consequence of Lemma \ref{critlem}. We proceed to prove (ii) and (iii). 
After normalization we have in particular
\[
\widehat F(0,x)=\begin{pmatrix} 0\\\widehat f(x)\end{pmatrix}
\]
with
\[
\widehat f(x)=\begin{pmatrix}-\omega x_2+(x_1^2+x_2^2)(\beta  x_1-\gamma x_2)\\
\omega x_1+(x_1^2+x_2^2)(\beta  x_2+\gamma x_1)\end{pmatrix}
\]
and $\psi(x)=x_1^2+x_2^2$ as well as 
\[
L_{\widehat f}(\psi)(x) =2\beta\psi(x)^2+\cdots,\quad L_{\widehat F}(\psi)(x) =2\alpha x_0\psi(x)^2+\cdots
\]
in view of \eqref{redeq}.
Since $\widehat f$ is a NFIM of $f$ up to degree three, there is a near-identity transformation
\[
x\mapsto \Gamma(x) =x + \text{h.o.t.},\quad \Gamma^{-1}(x) =x + \text{h.o.t.}
\]
such that
\[
\widehat f(x)=D\Gamma(x)^{-1}\,f(\Gamma(x)).
\]
Defining
\[
\widehat \psi:=\psi\circ\Gamma^{-1},
\]
the behavior of the Lie derivative under transformations implies that
\[
L_f(\widehat\psi)=L_{\widehat f}(\psi)\circ\Gamma^{-1}
\]
and we obtain
\[
L_f(\widehat\psi)=(2\beta\psi^2+\text{h.o.t.})\circ\Gamma^{-1}=2\beta\psi^2+\text{h.o.t.},
\]
as asserted in (ii). The proof of condition (iii) is similar. There is a near-identity normalizing transformation
\[
\begin{pmatrix}x_0\\x\end{pmatrix}\mapsto \Phi(\begin{pmatrix}x_0\\x\end{pmatrix}) =\begin{pmatrix}x_0\\x+ \text{h.o.t.}\end{pmatrix}
\]
such that
\[
\widehat F(\begin{pmatrix}x_0\\x\end{pmatrix})=D\Phi(\begin{pmatrix}x_0\\x\end{pmatrix})^{-1}\,F(\Phi(\begin{pmatrix}x_0\\x\end{pmatrix}))
\]
By construction $x_0\circ \Phi^{-1}=x_0$. Defining
\[
\widehat \theta=\psi\circ\Phi^{-1}.
\]
the same arguments as above show that
\[
L_F(\widehat\theta)=L_{\widehat F}(\psi)\circ\Phi^{-1}=2\alpha x_0\psi+ \text{h.o.t.}
\]
\end{proof}
\subsection{Computational matters I: Basic observations}
In this subsection we turn to critical parameter values and to the computation of their associated coefficients $\alpha$ and $\beta$. For a fixed critical parameter value $p^*$ we consider Taylor expansions
\begin{equation}\label{taylor}
\begin{array}{rcccccl}
f(x)&=& B(p^*)x&+&f^{(2)}(x)&+&f^{(3)}(x)+\cdots\\
F(x_0,x)&=& C(p^*)\begin{pmatrix}x_0\\x\end{pmatrix}&+& F^{(2)}(x_0,x)&+  &\cdots 
\end{array}
\end{equation}
with $f^{(i)}$, resp. $F^{(i)}$ homogeneous of degree $i$.
Note the block structure of
\[
C(p^*)=\begin{pmatrix} 0&0\\ 0&B(p^*)\end{pmatrix}.
\]
\begin{proposition}\label{betaprop}Let $p^*$ be a critical parameter value and let $\psi\not=0$ be a positive semidefinite quadratic form in the kernel of $L_{B(p^*)}$. Then $L_{B(p^*)}$ is invertible on $S_3$ and admits a one-dimensional kernel on $S_2$ and on $S_4$. The kernel is spanned by $\psi$ resp. by $\psi^2$.
\begin{itemize}
\item There are homogeneous $\widehat \psi_j$ of degree $j\in\{3,4\}$ and a scalar $\beta$ such that with $\widehat\psi=\psi+\widehat\psi_3+\widehat\psi_4$ one has
\[
L_f(\widehat\psi)= 2\beta\psi^2+\text{h.o.t.}
\]
\item This identity amounts to 
\[
L_{B(p^*)}(\psi)=0 \text{  in degree  } 2
\]
(which holds by definition),
\begin{equation}\label{betaeq3}
L_{B(p^*)}(\widehat\psi_3)+L_{f^{(2)}}(\psi)=0 \text{  in degree }3,
\end{equation}
and
\begin{equation}\label{betaeq4}
L_{B(p^*)}(\widehat\psi_4)+L_{f^{(2)}}(\widehat\psi_3) + L_{f^{(3)}}(\psi)=2\beta \psi^2  \text{  in degree }4.
\end{equation}
\item Equation \eqref{betaeq3} is a linear equation for $\widehat\psi_3$ which has a unique solution. 
 In equation \eqref{betaeq4} the right-hand side $2\beta\psi^2$ is the kernel component of the kernel-image decomposition of $L_{f^{(2)}}(\widehat\psi_3) + L_{f^{(3)}}(\psi)$ with respect to $L_{B(p^*)}$.
\end{itemize}
\end{proposition}
\begin{proof}
The NFIM of $F$ is given by \eqref{nfim}. The statements on invertibility and the dimension of the kernel follow from Remark \ref{skrem} and Lemma \ref{critlem}.
The remainder of the proof is a straightforward consequence of Proposition \ref{hopfconrestate} and elementary computations.
\end{proof}
\begin{remark} There is no need to compute $\widehat \psi_4$ in equation \eqref{betaeq4}; it suffices to determine the kernel component.

\end{remark}
Note that the computation of $\beta$ does not depend on $v\in \mathbb R^m$ in \eqref{odeeps}. One obtains $\alpha$ in a similar manner, but in this step the choice of $v$ is relevant. 
\begin{proposition}\label{alphaprop}Let $p^*$ be a critical parameter value such that the eigenvalue conditions from Lemma \ref{critlem}(b) hold, and let $\psi$ be as in Proposition \ref{betaprop}.
\begin{enumerate}[(a)]
\item There is a homogeneous $\widehat \theta_3$ of degree $3$ and a scalar $\alpha$ such that with $\widehat\theta=\psi+\widehat\theta_3$ one has
\[
L_f(\widehat\theta)= 2\alpha x_0\psi+\text{h.o.t.}
\]
\item This identity amounts to 
\[
L_{C(p^*)}(\psi)=0 \text{  in degree }2
\]
(which holds by definition) and
\begin{equation}\label{alphaeq3}
L_{C(p^*)}(\widehat\theta_3)+L_{F^{(2)}}(\psi)=2\alpha x_0\psi \text{  in degree }3.
\end{equation}
\item
Here $2\alpha x_0\psi$ is the kernel component of the kernel-image decomposition of $L_{F^{(2)}}(\psi) $ with respect to $L_{C(p^*)}$.
\end{enumerate}
\end{proposition}
\begin{proof} The proof is analogous to that of Proposition \ref{betaprop}.
We have an additional eigenvalue $\lambda_0=0$ here, therefore the kernel of $L_{C(p^*)}|_{S_3}$ is  two-dimensional, being spanned by $x_0^3$ and $x_0\psi$. But the kernel component of $L_{F^{(2)}}(\psi) $ cannot contain the monomial $x_0^3$. Indeed, since $0$ is a stationary point of system \eqref{odeeps} for all $\varepsilon$, in a representation
\begin{equation}\label{F2decomp}
F^{(2)}(x_0,x)=\begin{pmatrix}0\\f^{(2)}(x)\end{pmatrix}+x_0\begin{pmatrix}0\\Ax\end{pmatrix}+x_0^2\begin{pmatrix}0\\c\end{pmatrix}
\end{equation}
(with linear $A$ and constant $c$) one has necessarily $c=0$.
\end{proof}
\begin{remark} It is not necessary to compute $\widehat \theta_3$ in equation \eqref{alphaeq3}; only the kernel component of $L_{F^{(2)}}(\psi)$ is required. The latter task can be simplified further, since from \eqref{F2decomp} (noting $c=0$) and \eqref{alphaeq3} one sees that only the Lie derivative of $\psi$ with respect to $x_0\begin{pmatrix}0\\Ax\end{pmatrix}$ can contribute to the kernel component, and therefore it suffices to compute the kernel component $2\alpha \psi$ of $L_A(\psi)$.
\end{remark}


\begin{remark}\label{gradrem}
The determination of parameter values which give rise to a Hopf bifurcation starts with critical parameter values, which lie on a hypersurface in parameter space that is given by an equation $\omega(p)=0$ in parameter space; see the paragraph following Example \ref{polyex}. Here $\omega$ could be the Hurwitz determinant $\Delta_{n-1}$ or the constant term of the characteristic polynomial o ${L_{B}|_{S_{2}}}$. Given $p^*$ with $\omega(p^*)=0$ and $v\in\mathbb R^m$, we have
\[
\omega(p^*+\varepsilon v)=\varepsilon\left<{\rm grad}\,\omega(p^*),v\right>+o(\varepsilon)
\]
for any (small) $\varepsilon\in\mathbb R$. Hence there exist $v$ such that $\alpha\not=0$ for system \eqref{odeeps} whenever ${\rm{grad}}\,\omega(p^*)\not=0$.
\end{remark}

\subsection{Computational matters II: The procedure}
We start again with system \eqref{odegen}, with $h(0,p)=0$ and $B(p)=Dh(0,p)$ for all $p$.
\begin{enumerate}[1.]
\item Annihilating polynomials for $L_{B(p)}|_{S_k}$ for $k\in\{1,2,3,4\}$: 
 The case $k=1$ amounts to finding the minimal polynomial (or the characteristic polynomial) of $B(p)$; from this one successively determines annihilating polynomials on $S_2$, $S_3$ and $S_4$, as described in \cite{MSW}, Section 3.\\
 Note: If one starts from an annihilating polynomial $\mu$ for $L_C$  (with any linear $C$) on $S_1$ then the coefficients of the corresponding anihilating polynomials for $L_C$ on any $S_k$ are polynomials in the coefficients of $\mu$; see \cite{MSW}.
\item Basic tasks (see \cite{MSW}, Prop. 2.1): We abbreviate $V=S_k$ and $T:=L_{B(p)}|_{S_k}$, and let 
\[
q(\tau)=\tau^\ell +\sum_{i=1}^\ell \beta_i\tau^{\ell-i}
\]
be a polynomial such that $q(T)=0$, with the additional condition that either $\beta_\ell\not=0$ (whenever $T$ is invertible) or $\beta_\ell=0\not=\beta_{\ell-1}$.
\begin{itemize}
\item If $\beta_\ell\not=0$ then 
 the solution of the equation $Tv=w$ is given by
\[
w=-\frac{1}{\beta_\ell}(T^{\ell-1}+\sum_{i=1}^{\ell-1}\beta_i T^{\ell-1-i})v.
\]
\item If $\beta_\ell=0$ and $\beta_{\ell-1}\not=0$ then the kernel component of $v$ is given by
\[
\frac{1}{\beta_{\ell-1}}(T^{\ell-2}+\sum_{i=1}^{\ell-2}\beta_i T^{\ell-2-i})v.
\]
\end{itemize}
\item Critical parameter values:  Note that all coefficients of annihilating polynomials are themselves polynomials in the parameters $p$. 
\begin{itemize}
\item Choose an annihilating polynomial for $L_{B(p)}|_{S_1}$ with generically nonzero constant coefficient $\gamma_1(p)$. (If the constant coefficient $\gamma_1$ is zero for any $p$ then no Hopf bifurcation can occur.)
\item Choose an annihilating polynomial for $L_{B(p)}|_{S_2}$ with constant coefficient $\gamma_2(p)$ and degree one coefficent $\gamma_3(p)$. The critical parameters are those zeros of $\gamma_2$ which are not zeros of $\gamma_1$ or $\gamma_3$. (If those conditions cannot be satisfied for any annihilating polynomial then no Hopf bifurcation can occur.)
\end{itemize}
\item Check conditions: Given a critical parameter value $p^*$, verify that the conditions from Lemma \ref{critlem}(c) hold, and determine a positive semidefinite quadratic form $\psi$ in the kernel of $L_{B(p)}|_{S_2}$.
\item Determining $\beta$ for a given $p^*$: Solve equation \eqref{betaeq3} for $\widehat \psi_3$, then determine the kernel component of $L_{f^{(2)}}(\widehat\psi_3) + L_{f^{(3)}}(\psi)$ and divide by $2\psi^2$. (To reduce the expenditure, in the last step one may specialize $x$ to any $\overline x\in \mathbb R^n$ with $\psi(\overline x)\not=0$.)
\item Determining $\alpha$ for a given $p^*$: With the notation of \eqref{taylor}, first define $A$ via
\[
\begin{pmatrix}0\\Ax\end{pmatrix}= F^{(2)}(1,x)-\begin{pmatrix}0\\f^{(2)}(x)\end{pmatrix}
\]
(compare \eqref{alphaeq3}). Then determine the kernel component of $L_A(\psi)$ and divide by $2\psi$. (Again one may specialize $x$ to $\overline x$ to reduce expenditure.)
\end{enumerate}

We add some comments on practical matters and the size of the problems.
In order to compute the coefficients $\beta$ and $\alpha$, it is necessary to first obtain annihilating polynomials of $L_{B}|_{S_{k}}$ for $k\in \{2,3,4\}$. Some simplifications and shortcuts are noted in \cite{Kru}, and furthermore, this work needs to be done only once (see the Example below). Starting from the characteristic polynomial, the degree of an annihilating polynomial of $L_{B}|_{S_{k}}$ is given by the dimension  ${n+k-1}\choose {n-1}$ of the space of homogeneous polynomials of degree $k$ in $n$ variables, which is tolerable for reasonably small $n$. (For $n=3$ we have dimensions $6$, $10$ and $15$, respectively; for $n=4$ we have dimensions $10$, $20$ and $35$.) The presence of parameters complicates the problem, and this is the main reason to work with the annihilating polynomials, instead of more direct methods.
\begin{example}Given the characteristic polynomial $\chi=\tau^3+c_1\tau^2+c_2\tau+c_3$, and $c_1c_2=c_3$, we determine annihilating polynomials $\chi_k$ (in a factorized form) of $L_{B}|_{S_{k}}$ for $k=2,3,4$. 
\begin{equation*}\begin{split}
 \chi_2=&\tau(\tau^2+4c_{2})(\tau+2c_1)(\tau^2+2c_1^2\tau+c_1^2+c_2)\\
 &\ \\
 \chi_3=&(\tau^2+9c_2)(\tau+3c_1)(\tau+c_1)(\tau^2+c_2)\\
                    \cdot & (\tau^4+6c_1\tau^3+(13c_2^2+5c_2)\tau^2+(12c_1^3+18c_1c_2)\tau+4c_1^4+17c_1^2c_2+4c_2^2)\\
 &\ \\
  \chi_4=& \tau(\tau^2+16c_2)(\tau+4c_1)(\tau^2+4c_1\tau+4c_1^2+4c_2)(\tau+2c_1)\\
                    \cdot & (\tau^2+4c_2)(\tau^2+2c_1 \tau+c_1^2+c_2)\\
                    \cdot & (\tau^4+8c_1 \tau^3+(22c_1^2+10c_2)\tau^2+(24c_1^3+56c_1c_2)\tau+9c_1^4+82c_1^2c_2+9c_2^3)
\end{split}\end{equation*}
\end{example}



\section{A general approach for polynomial systems}
For system \eqref{odegen} we imposed the restriction that $0$ is stationary for any value of $p$. For general parameter dependent systems this restriction is not welcome, and one would prefer an approach which allows to determine those parameter values for which a Hopf bifurcation occurs at some stationary point. We will outline here how this can be done at least for polynomial systems. The fundamental idea in this respect is due to Errami et al. \cite{EEGSSW} who focussed attention on chemical reaction networks and used methods from real algebra. Here we will outline a variant for arbitrary polynomial systems which focuses on explicit conditions for the parameter values and is similar in spirit to the determination of Tikhonov-Fenichel parameter values in \cite{GWZ}). We consider a system
\begin{equation}\label{odepoly}
\dot x= q(x,p)=\begin{pmatrix} q_{1}(x,p)\\\vdots\\ q_{n}(x,p)\end{pmatrix}
\end{equation}
with polynomial right hand side (in variables and parameters), but no fixed stationary point. We are interested in parameter values such that a Hopf bifurcation occurs at some stationary point of \eqref{odepoly}, and by elimination theory we can determine necessary conditions for these parameter values.  For background and terminology see Cox et al. \cite{CLOS} (in particular Chapter 3).
\begin{proposition}\label{supergen}
For $y\in\mathbb R^n,\,p\in\mathbb R^m$ let $\omega(y,p)=\Delta_{n-1}(Dq(y,p))$ be the Hurwitz determinant of the Jacobian of $q$.
\begin{enumerate}[(a)]
\item If $y^*$ is a stationary point of $\dot x=q(x,p^*)$ from which a Hopf bifurcation emanates then $(y^*,\,p^*)$ is a zero of the ideal
 \[
  I:=\langle q_{1}(y,p),\dots,q_{n}(y,p),\,\omega(y,p) \rangle\subseteq \mathbb{R}[y,p].
 \]
\item Moreover $p^*$ is then a zero of the elimination ideal  
\[
I\cap \mathbb{R}[p].
\]
\end{enumerate}
\end{proposition}
\begin{proof} In order for $y^*$ to be stationary for \eqref{odepoly} at $p=p^*$, $(y^*,p^*)$ must be a common zero of $q_1,\ldots, q_n$. The condition $\omega(y^*,p^*)=0$ is due to Lemma \ref{critlem}. Thus part (a) holds, and part (b) is a direct consequence.
\end{proof}

\begin{remark}
\begin{itemize}
\item The underlying idea is that $q_1(y,p)=\cdots=q_n(y,p)=\omega(y,p)=0$, for fixed $p$, is an overdetermined system for $y$, and its solvability will in general impose nontrivial conditions on $p$ which are given by the elimination ideal.
\item Taking a suitable elimination order, such as the lexicographical ordering with respect to the variables $y_{1},\dots,y_{n}, p_1,\ldots, p_m$, one directly obtains a Groebner basis of the intersection
\[
I\cap \mathbb{R}[p] 
\]
from a Groebner basis of $I$; see Cox et al. \cite{CLOS}.
\item  One may alternatively consider an annihilating polynomial for the action of $L_{Dq(y,p)}$ on $S_2$ and take $\omega(y,p)$ as its constant coefficient; see Remark \ref{liucondrem}.
\item One should keep in mind that computing the elimination ideal and investigating the polynomial conditions in the parameters $p_{1},\ldots,p_{m}$ is just the first step.
For instance, if we consider the case $n=3$ we have to ensure that $c_1>0$, $c_{2}>0$ (in the notation of \eqref{polygen}) which imposes further conditions on the parameters. 
\end{itemize}
\end{remark}


\section{Applications}

\noindent In this section we discuss some examples which, in addition to illustrating the feasibility of the algorithm, are of interest for applications. The systems are dependent on parameters. While the computations go through with the full set of parameters, the output in some cases is too large to be conveniently readable; for this reason we will then specialize parameters.

\subsection{Dimension two}
 We first consider (as a kind of benchmark problem) a general two-dimensional  vector field
 \begin{align*}
  h(x)=\begin{pmatrix}0 & -d \\ 1 & s\end{pmatrix}\cdot \begin{pmatrix} x_{1} \\ x_{2}\end{pmatrix}+ h_2(x)+h_3(x)+\cdots
 \end{align*}
with
stationary point $0$ and linearization $B$ in Frobenius normal form. (Note that a Frobenius normal form can be obtained by rational operations.) From the characteristic polynomial
\begin{align*}
 \chi_{B}(\tau)=\tau^2-s\tau+d
\end{align*}
the conditions $d>0$ and $s=0$ for critical parameter values can be read off directly.  We allow general quadratic and cubic terms in the Taylor expansion, thus
\[
h_2(x)=\begin{pmatrix}c_{1}x_{1}^2+c_{2} x_{1}x_{2}+c_{3}x_{2}^2\\ d_{1} x_{1}^2+d_{2}x_{1}x_{2}+d_{3}x_{2}^2\end{pmatrix}
\]
and
\[
h_3(x)=\begin{pmatrix}a_1x_1^3+a_2x_1^2x_2+a_3x_1x_2^2+a_4x_2^3\\ b_1x_1^3+b_2x_1^2x_2+b_3x_1x_2^2+b_4x_2^3 \end{pmatrix}
\]
where $a_{i},b_{i}, c_{i},d_{i} \in \mathbb{R}$. 
With $d>0$ and $s=0$ the characteristic polynomial of $L_{B}|_{S_{2}}$ is equal to
\begin{align*}
 \chi_2(\tau)=(\tau^2+4d)\cdot \tau,
\end{align*}
and the kernel of $L_{B}|_{S_{2}}$ is spanned by the positive definite invariant $\psi_{2}:=x_1^2+dx_2^2$. Applying our algorithm, we get the reduced system $\dot z=\beta z^2$, with
\begin{align*}
 \beta=\frac{(-2c_1d_1d^2-d_1d_2d^2+3a_1d^2+b_2d^2+c_1c_2d-d_2d_3d+a_3d+3b_4d+c_2c_3+2c_3d_3)}{4d^2}
\end{align*}
In order to illustrate the computation of the coefficient $\alpha$ for a perturbation of the critical parameter value $p^*$ we keep the $a_i,b_i,c_i,d_i$ constant and consider only  $p^{*}=(d,0)$, with $v=(v_{1},v_{2})$. Thus from 
\begin{align*}
 F_{v}(x_{0},x)=\begin{pmatrix}0 \\ h(x,p^{*}+x_{0}v)\end{pmatrix}
\end{align*}
we obtain
\begin{align*}
A=v_{2}\cdot
 \begin{pmatrix}
  0 & 1 \\ 0 & 1
 \end{pmatrix}.
\end{align*}
The kernel component of $L_{A}(\psi_{2})=2y(v_{1}x_1+dv_{2}x_{2})$ with respect to $L_{B}$ is equal to $v_{2}\cdot \psi_{2}$. Consequently, we have $\alpha=v_{2}$. 

\subsection{The FitzHugh-Nagumo system}
This system (see FitzHugh \cite{Fitz}, Nagumo et al. \cite{Nagumo} and Murray \cite{Mur}, Ch.~7.5)
is a simplified version of the Hodgkin-Huxley system which models the firing of a neuron. From several versions we choose
\begin{align*}
 &\dot x_1=x_1-\frac{x_1^3}{3}-x_2+I&=:q_{1}\\
 &\dot x_2=cx_1-bx_2+a&=:q_{2}
\end{align*}
with parameters $a,b,c,I$. (This is  FitzHugh's \cite{Fitz} original system (1),(2), with $y$ replaced by $-cy$ and variables and some parameters renamed). By Proposition \ref{supergen} we determine necessary conditions on those parameter values for which a Hopf bifurcation emanates from some stationary point.
Thus we have $q_1(y,p)=q_2(y,p)=0$ and the further condition $0=\text{trace}(Dq(y))=1-y_1^2-b=\omega(y,p)$ to impose on the parameter values. Eliminating the variables $y_1,y_2$ in the ideal
\[
 J=\langle q_{1},q_{2},\omega\rangle
\]
 yields the necessary algebraic condition 
 \[
 b^5+3b^4-6b^3c+9b^2I^2-18abI-6b^2c+9bc^2+9a^2-4b^2+12bc-9c^2=0.
 \]
From here on we restrict attention to the case with
\[
0<b<1 \text{  and  } b<c.
\]
Thus, in contrast to most discussions (e.g. in the references cited above) we focus on a parameter region where the $x_1$-nullcline is the graph of a strictly decreasing function, and no slow-fast oscillations are to be expected.
We rewrite the condition as
\[
 I=\frac{(3a\pm \sqrt{-b^5-3b^4+6b^3c+6b^2c-9bc^2+4b^2-12bc+9c^2}}{3b}
\]
With the above restriction on the parameters the unique stationary point is given by 
\[
y_1^*=\sqrt{1-b} \text{  and  }y_2^*=\frac{1}{b}\cdot\left(a+c\cdot \sqrt{1-b}\right).
\]
For both choices of sign in the expression for $I$ we obtain the quadratic invariant $\psi=-2bxy+cx^2+y^2$ which is positive definite (and thus yields a pair of purely imaginary eigenvalues) whenever $b^2<c$. In this case we obtain the reduced system
\[
 \dot z=-\frac{(b^2-2b+c)}{(4(b^2-c)^2)}z^2.
\]
Note that the sign of $\beta=-(b^2-2b+c)/(4(b^2-c)^2)$ may be positive or negative. For $c\geq 1$ it is certainly negative and we have an attracting limit cycle.


\subsection{A predator-prey system}
 The three dimensional system 
  \begin{equation}\label{prepre}
\begin{array}{rcl}
\dot x_1&=&rx_1(1-x_1)-\gamma x_1x_3\\
\dot x_2&=&-\eta x_2+\delta x_1x_3\\
\dot x_3&=&bx_2-dx_3+\eta x_2-\delta x_1x_3
\end{array}
\end{equation}
was introduced in \cite{KLLW}. It models a predator-prey population with one prey $x_1$ and two stages $x_2$, $x_3$ for the predator. All parameters are $>0$, and only nonnegative solutions are of interest. \\
We first use Proposition \ref{supergen} to determine parameter values for which a Hopf bifurcation may occur at some stationary point. The elimination ideal (computed with the help of {\sc {Singular}} \cite{singular}) is generated by one element
\[
h=h_{1}\cdot h_{2}\cdot h_{3}^2\cdot h_{4}\cdot h_{5}\cdot h_{6}\cdot h_{7}\cdot h_{8}^2\cdot h_{9}^2
\]
with factors 
\begin{align*}
h_1&=\delta+d+\eta\\
h_{2}&=d+\eta\\
h_{3}&=d\\
h_{4}&=-\delta^2 b^3-\delta^2 b^2 d+\delta b^2 d^2-\delta^2 b^2 \eta-\delta^2 b d \eta+3 \delta b^2 d \eta\\
&+3 \delta b d^2 \eta+\delta b^2 \eta^2+3 \delta b d \eta^2+2 \delta d^2 \eta^2-\delta b d \eta r+b d^2 \eta r+b d \eta^2 r+2 d^2 \eta^2 r\\
h_{5}&=-\delta b+d \eta+\delta r+d r+\eta r+r^2\\
h_{6}&=-d+r\\
h_{7}&=-\eta+r\\
h_{8}&=r\\
h_{9}&=\eta.
\end{align*}
The equation $h=0$ defines a hypersurface on parameter space and for each $i$, $h_i=0$ defines an irreducible component of this hypersurface. Hopf bifurcations correspond to curves in parameter space which transversally cross some component. The nonnegativity requirement for parameters implies that $h_1,\,h_2,\,h_3,\,h_8$ and $h_9$ are irrelevant for potential Hopf bifurcations, and one verifies that $h_6$ and $h_7$ correspond to the stationary point $0$ (at which no Hopf bifurcation can take place). This leaves $h_4$ and $h_5$ for
possible Hopf bifurcations at the interior stationary point 
\begin{align*}
 y_1&=\frac{d\eta}{b\delta}\\
 y_2&=\frac{dr}{b\gamma}\left(\frac{b\delta-d\eta}{b\delta}\right)\\
 y_3&=\frac{r}{\gamma}\left(\frac{b\delta-d\eta}{b\delta}\right)
\end{align*}
which exists whenever $b\delta-d\eta>0$. One verifies that $h_4=0$ indeed corresponds to Liu's condition $\Delta_2=0$ at the interior stationary point. Moreover we may set $\gamma=1$ with no loss of generality (since one may scale system \eqref{prepre} via $x_2\mapsto\gamma x_2$, $x_3\mapsto\gamma x_3$), and we will do so in the following.\\
At the interior stationary point, the characteristic polynomial of the Jacobian is given by
$\chi_{B}(\tau)=\tau^3+c_{1}\tau^2+c_{2}\tau+c_{3}$ with 
\begin{align*}
 c_{1}&=\frac{bd\delta+b\delta\eta+d\delta\eta+d\eta r}{b\delta}\\
 c_{2}&=\frac{d\eta r(bd-b\delta+b\eta+2d\eta)}{b^2\delta}\\
 c_{3}&=\frac{d\eta r(b\delta-d\eta)}{b\delta}.
\end{align*}
The condition $h_{4}=0$ is linear in the parameter $r$ and equivalent to
\begin{align*}
 r&=\frac{\delta}{{d\eta(bd-b\delta+b\eta+2d\eta)}}\\
 &\cdot \left(b^3\delta-b^2d^2+b^2d\delta-3b^2d\eta+b^2\delta\eta-b^2\eta^2-3bd^2\eta+bd\delta\eta-3bd\eta^2-2d^2\eta^2\right)
\end{align*}
as long as the denominator does not vanish; this allows to substitute the right hand side for $r$.\\

For further investigation and application of the results from Section \ref{unsuitsec}, we perform a translation, replacing $x_i$ by $y_i+x_i$ such that $(0,0,0)$ is a stationary point of the transformed system. (We keep old names here for new variables.) The following computations are carried out with the help of the software system {\sc{Maple}}.
A quadratic invariant $\psi$ of the linearization (see Lemma \ref{critlem}) is given by
\begin{align*}
  \psi=a_1 x_1^2+a_2 x_1x_2+a_3 x_1x_3+a_4 x_2^2+a_5 x_2x_3+a_6 x_3^2
\end{align*}
with
\begin{align*}
a_{1}&=\delta^2(b^5\delta^2-b^4d^2\delta+2b^4d\delta^2-4b^4d\delta\eta+b^4\delta^2\eta-b^4\delta\eta^2\\
&-b^3d^3\delta+b^3d^3\eta+b^3d^2\delta^2-8b^3d^2\delta\eta+3b^3d^2\eta^2+2b^3d\delta^2\eta-5b^3d\delta\eta^2\\
&+b^3d\eta^3+b^2d^4\eta-4b^2d^3\delta\eta+6b^2d^3\eta^2+b^2d^2\delta^2\eta-7b^2d^2\delta\eta^2+4b^2d^2\eta^3\\
&+3bd^4\eta^2-3bd^3\delta\eta^2+5bd^3\eta^3+2d^4\eta^3),\\
a_{2}&=2\delta b\eta d(b^3\delta-b^2d^2+b^2d\delta-3b^2d\eta\\
&+b^2\delta\eta-b^2\eta^2-3bd^2\eta+bd\delta\eta-3bd\eta^2-2d^2\eta^2),\\
a_{3}&=-2\delta\eta d^2(b^3\delta-b^2d^2+b^2d\delta-3b^2d\eta+b^2\delta\eta\\
&-b^2\eta^2-3bd^2\eta+bd\delta\eta-3bd\eta^2-2d^2\eta^2),\\
a_{4}&=(b+\eta)b^2\eta^2d^2,\\
a_{5}&=2b^2\eta^3d^2,\\
a_{6}&=(b^2\delta-bd^2+bd\delta-3bd\eta-2d^2\eta)\eta^2d^2.\\
\end{align*}
Using this invariant, we compute a truncated reduced system (corresponding to \eqref{betacon}) of the form $\dot z=\beta\cdot z^2+\cdots$ where $\beta$ depends on the parameters $b,d,\delta,\eta$. The general expression for the coefficient $\beta$ is given in the Appendix; one sees that the output is manageable if a bit unwieldy (which is caused by the nature of the problem, not by some arbitrary choice).\\
The missing conditions from Liu \cite{Liu} for a Hopf bifurcation are $c_{1}>0,\,c_{2}>0$. We only pick one special case here, setting
$b=d=r$, and
\[
 \delta=\frac{d^2+5d\eta+6\eta^2+\sqrt{d^4+18d^3\eta+69d^2\eta^2+84d\eta^3+36\eta^4}}{4d+4\eta}
\]
from $h_4=0$.
Choosing $\eta=d=1$, which guarantees existence of the interior stationary point due to $b\delta=\frac{1}{2}\left(3+\sqrt{13}\right)>1=d\eta$, and also guarantees $c_1>0$ and $c_2>0$ (as well as positive semidefiniteness of $\psi$), one finds the coefficient $\beta\approx- 0.03<0$. Altogether, there is a Hopf bifurcation with an attracting limit cycle. (By  Remark \ref{gradrem} we have $\alpha\neq 0$ for some direction $v$.)

\appendix
\section*{Appendix: A coefficient}
We did not write down the parameter dependent coefficient $\beta=\beta(b,d,\delta,\eta)$ in the predator-prey system above, since it takes up some space.  It has the form $\beta=\frac{\mu}{\nu}$ with numerator
\begin{equation*}\begin{split}
\mu=&b^2(b+\eta )\cdot (b^{10}d\delta ^4-b^{10}\delta ^5+b^{10}\delta ^4\eta -3b^9d^3\delta ^3+6b^9d^2\delta ^4-13b^9d^2\delta ^3\eta\\
&-3b^9d\delta ^5+18b^9d\delta ^4\eta -13b^9d\delta ^3\eta ^2-3b^9\delta ^5\eta +6b^9\delta ^4\eta ^2-3b^9\delta ^3\eta ^3\\
&+2b^8d^5\delta ^2-6b^8d^4\delta ^3+3b^8d^4\delta ^2\eta +6b^8d^3\delta ^4-21b^8d^3\delta ^3\eta +4b^8d^3\delta ^2\eta ^2\\
&-2b^8d^2\delta ^5+19b^8d^2\delta ^4\eta -45b^8d^2\delta ^3\eta ^2+4b^8d^2\delta ^2\eta ^3-b^8d\delta ^5\eta\\
&+19b^8d\delta ^4\eta ^2-21b^8d\delta ^3\eta ^3+3b^8d\delta ^2\eta ^4-2b^8\delta ^5\eta ^2+6b^8\delta ^4\eta ^3\\
&-6b^8\delta ^3\eta ^4+2b^8\delta ^2\eta ^5+16b^7d^6\delta \eta -42b^7d^5\delta ^2\eta +97b^7d^5\delta \eta ^2\\
&+42b^7d^4\delta ^3\eta -256b^7d^4\delta ^2\eta ^2+209b^7d^4\delta \eta ^3-22b^7d^3\delta ^4\eta +221b^7d^3\delta ^3\eta ^2\\
&-410b^7d^3\delta ^2\eta ^3+209b^7d^3\delta \eta ^4+6b^7d^2\delta ^5\eta -84b^7d^2\delta ^4\eta ^2+221b^7d^2\delta ^3\eta ^3\\
&-256b^7d^2\delta ^2\eta ^4+97b^7d^2\delta \eta ^5+6b^7d\delta ^5\eta ^2-22b^7d\delta ^4\eta ^3+42b^7d\delta ^3\eta ^4\\
&-42b^7d\delta ^2\eta ^5+16b^7d\delta \eta ^6-8b^6d^8\eta +36b^6d^7\delta \eta -58b^6d^7\eta ^2-64b^6d^6\delta ^2\eta \\
&+368b^6d^6\delta \eta ^2-185b^6d^6\eta ^3+56b^6d^5\delta ^3\eta -680b^6d^5\delta ^2\eta ^2+1273b^6d^5\delta \eta ^3\\
&-321b^6d^5\eta ^4-24b^6d^4\delta ^4\eta +524b^6d^4\delta ^3\eta ^2-2019b^6d^4\delta ^2\eta ^3+1872b^6d^4\delta \eta ^4\\
&-321b^6d^4\eta ^5+4b^6d^3\delta ^5\eta -166b^6d^3\delta ^4\eta ^2+1105b^6d^3\delta ^3\eta ^3-2019b^6d^3\delta ^2\eta ^4\\
&+1273b^6d^3\delta \eta ^5-185b^6d^3\eta ^6+12b^6d^2\delta ^5\eta ^2-166b^6d^2\delta ^4\eta ^3+524b^6d^2\delta ^3\eta ^4\\
&-680b^6d^2\delta ^2\eta ^5+368b^6d^2\delta \eta ^6-58b^6d^2\eta ^7+4b^6d\delta ^5\eta ^3-24b^6d\delta ^4\eta ^4\\
&+56b^6d\delta ^3\eta ^5-64b^6d\delta ^2\eta ^6+36b^6d\delta \eta ^7-8b^6d\eta ^8-84b^5d^8\eta ^2+308b^5d^7\delta \eta ^2\\
&-528b^5d^7\eta ^3-424b^5d^6\delta ^2\eta ^2+2110b^5d^6\delta \eta ^3-1402b^5d^6\eta ^4+264b^5d^5\delta ^3\eta ^2\\
&-2658b^5d^5\delta ^2\eta ^3+5096b^5d^5\delta \eta ^4-1914b^5d^5\eta ^5-68b^5d^4\delta ^4\eta ^2+1270b^5d^4\delta ^3\eta ^3\\
&-4776b^5d^4\delta ^2\eta ^4+5096b^5d^4\delta \eta ^5-1402b^5d^4\eta ^6+4b^5d^3\delta ^5\eta ^2-182b^5d^3\delta ^4\eta ^3\\
&+1270b^5d^3\delta ^3\eta ^4-2658b^5d^3\delta ^2\eta ^5+2110b^5d^3\delta \eta ^6-528b^5d^3\eta ^7+4b^5d^2\delta ^5\eta ^3\\
&-68b^5d^2\delta ^4\eta ^4+264b^5d^2\delta ^3\eta ^5-424b^5d^2\delta ^2\eta ^6+308b^5d^2\delta \eta ^7-84b^5d^2\eta ^8\\
&-364b^4d^8\eta ^3+1040b^4d^7\delta \eta ^3-1918b^4d^7\eta ^4-1032b^4d^6\delta ^2\eta ^3+5182b^4d^6\delta \eta ^4\\
&-4034b^4d^6\eta ^5+400b^4d^5\delta ^3\eta ^3-4110b^4d^5\delta ^2\eta ^4+8532b^4d^5\delta \eta ^5-4034b^4d^5\eta ^6\\
&-44b^4d^4\delta ^4\eta ^3+986b^4d^4\delta ^3\eta ^4-4110b^4d^4\delta ^2\eta ^5+5182b^4d^4\delta \eta ^6-1918b^4d^4\eta ^7\\
&-44b^4d^3\delta ^4\eta ^4+400b^4d^3\delta ^3\eta ^5-1032b^4d^3\delta ^2\eta ^6+1040b^4d^3\delta \eta ^7-364b^4d^3\eta ^8\\
&-832b^3d^8\eta ^4+1728b^3d^7\delta \eta ^4-3504b^3d^7\eta ^5-1088b^3d^6\delta ^2\eta ^4+6024b^3d^6\delta \eta ^5\\
&-5416b^3d^6\eta ^6+192b^3d^5\delta ^3\eta ^4-2504b^3d^5\delta ^2\eta ^5+6024b^3d^5\delta \eta ^6-3504b^3d^5\eta ^7\\
&+192b^3d^4\delta ^3\eta ^5-1088b^3d^4\delta ^2\eta ^6+1728b^3d^4\delta \eta ^7-832b^3d^4\eta ^8-1056b^2d^8\eta ^5\\
&+1408b^2d^7\delta \eta ^5-3304b^2d^7\eta ^6-416b^2d^6\delta ^2\eta ^5+3048b^2d^6\delta \eta ^6-3304b^2d^6\eta ^7\\
&-416b^2d^5\delta ^2\eta ^6+1408b^2d^5\delta \eta ^7-1056b^2d^5\eta ^8-704bd^8\eta ^6+448bd^7\delta \eta ^6\\
&-1440bd^7\eta ^7+448bd^6\delta \eta ^7-704bd^6\eta ^8-192d^8\eta ^7-192d^7\eta ^8)\cdot(bd-b\delta +b\eta +2d\eta )^2
\end{split}\end{equation*}
and denominator
\begin{align*}
 \nu=&4d^2\eta ^2(b\delta -d\eta )\\
 \cdot &(b^6\delta ^2+b^5d^2\delta -2b^5d\delta ^2+b^5\delta ^3-2b^5\delta ^2\eta +b^5\delta \eta ^2-b^4d^4+3b^4d^3\delta -5b^4d^3\eta \\
 &-3b^4d^2\delta ^2+15b^4d^2\delta \eta -7b^4d^2\eta ^2+b^4d\delta ^3-11b^4d\delta ^2\eta+15b^4d\delta \eta ^2-5b^4d\eta ^3\\
 &+b^4\delta ^3\eta-3b^4\delta ^2\eta ^2+3b^4\delta \eta ^3-b^4\eta ^4-7b^3d^4\eta +15b^3d^3\delta \eta -25b^3d^3\eta ^2\\
 &-9b^3d^2\delta ^2\eta +38b^3d^2\delta \eta ^2-25b^3d^2\eta ^3+b^3d\delta ^3\eta-9b^3d\delta ^2\eta ^2\\
 &+15b^3d\delta \eta ^3-7b^3d\eta ^4-18b^2d^4\eta ^2+24b^2d^3\delta \eta ^2-40b^2d^3\eta ^3-6b^2d^2\delta ^2\eta ^2\\
 &+24b^2d^2\delta \eta ^3-18b^2d^2\eta ^4-20bd^4\eta ^3+12bd^3\delta \eta ^3-20bd^3\eta ^4-8d^4\eta ^4)^2\\
 \cdot &(b^6\delta ^2+4b^5d^2\delta -8b^5d\delta ^2+6b^5d\delta \eta +4b^5\delta ^3-8b^5\delta ^2\eta \\
 &+4b^5\delta \eta ^2-4b^4d^4+12b^4d^3\delta -20b^4d^3\eta-12b^4d^2\delta ^2+60b^4d^2\delta \eta\\
 &-31b^4d^2\eta ^2+4b^4d\delta ^3-44b^4d\delta ^2\eta +60b^4d\delta \eta ^2-20b^4d\eta ^3+4b^4\delta ^3\eta\\
 &-12b^4\delta ^2\eta ^2+12b^4\delta \eta ^3-4b^4\eta ^4-28b^3d^4\eta+60b^3d^3\delta \eta -100b^3d^3\eta ^2-36b^3d^2\delta ^2\eta \\
 &+152b^3d^2\delta \eta ^2-100b^3d^2\eta ^3+4b^3d\delta ^3\eta -36b^3d\delta ^2\eta ^2+60b^3d\delta \eta ^3-28b^3d\eta ^4\\
 &-72b^2d^4\eta ^2+96b^2d^3\delta \eta ^2-160b^2d^3\eta ^3-24b^2d^2\delta ^2\eta ^2+96b^2d^2\delta \eta ^3-72b^2d^2\eta ^4\\
 &-80bd^4\eta ^3+48bd^3\delta \eta ^3-80bd^3\eta ^4-32d^4\eta ^4).
\end{align*}

\bigskip\noindent{\bf Acknowledgement.} The first author acknowledges support by the DFG Research Training Group GRK 1632 ``Experimental and constructive algebra''.


\end{document}